\ifdefined\JNTcompact
  \documentclass[final,3p,10pt,number]{elsarticle}
\else
  \documentclass[preprint,12pt,number]{elsarticle}
\fi
\usepackage[T1]{fontenc}
\usepackage{lmodern}
\usepackage{amsmath,amssymb,amsthm,mathtools}
\usepackage{microtype,booktabs,array,needspace,float}
\usepackage[hidelinks]{hyperref}
\numberwithin{equation}{section}
\newtheorem{theorem}{Theorem}[section]
\newtheorem{proposition}{Proposition}[section]
\newtheorem{lemma}{Lemma}[section]
\newtheorem{corollary}{Corollary}[section]
\newcommand{\dd}{\,\mathrm{d}}
\newcommand{\R}{\mathbb R}
\newcommand{\Q}{\mathbb Q}
\newcommand{\Z}{\mathbb Z}
\newcommand{\A}{\mathbb A}
\newcommand{\GL}{\mathrm{GL}}
\newcommand{\SL}{\mathrm{SL}}
\DeclareMathOperator{\supp}{supp}
\hypersetup{pdftitle={Zero density in short strips and short-interval bounds toward Hypothesis S},
 pdfauthor={Yangbo Ye},pdfsubject={Manuscript prepared for Journal of Number Theory}}
\journal{Journal of Number Theory}

\begin{document}
\begin{frontmatter}
\title{Zero density in short strips and short-interval bounds toward Hypothesis S}
\author{Yangbo Ye}
\ead{yangbo-ye@uiowa.edu}
\affiliation{organization={Department of Mathematics, The University of Iowa},
 city={Iowa City}, state={IA}, postcode={52242}, country={USA}}

\begin{abstract}
We deduce zero-density estimates in short vertical strips from
short-interval moments of globally tempered automorphic $L$-functions,
by adapting the large-value estimate of Guth and Maynard and the
Type II zero argument of Maynard and Pratt to the framework of Ye
and Zhang. For zeta we obtain
$N_\zeta(\sigma;T,T+T^a)\ll T^{15a(1-\sigma)/(3+5\sigma)+\varepsilon}$
for $2/3\le a\le1$ and $7/10\le\sigma\le4/5$, improving
Ye--Zhang's exponent in the interior of this range.
Combining the localized density bound with the recent quasi-Riemann
theorems, we estimate smoothed prime-power sums with phase
$\exp(2\pi i\alpha n^\beta)$ on intervals of length $K^\delta$.
For $0<\beta<13/20$, we obtain a power saving whenever
$\max\{17/30,1-\beta\}<\delta<1$; in particular the square-root
phase admits every $\delta>17/30$. A revised transform estimate
retains both the stationary spread and the Fourier width $K/L$.
For example, the bound at $\beta=1/2$, $\delta=3/5$ is
$K^{15/26+\varepsilon}$, uniformly for frequencies $m\alpha$
with $1\le m\le K^{1/10}$.
We also retain power savings for certain self-dual $\GL(2)$ twists
and self-dual $\GL(3)$ cusp forms, without assuming Ramanujan.
All three applications have explicit frequency dependence.
\end{abstract}

\end{frontmatter}

\section{Introduction}

Zero-density estimates connect the mean values of $L$-functions with
the distribution of primes and cancellation in arithmetic sums.
Throughout, $\varepsilon>0$ is arbitrary.
Our main results refine the short-strip framework of Ye--Zhang
\cite{YZ}, using Guth--Maynard's large-value estimate \cite{GM} and
the Type II integral argument of Maynard--Pratt
\cite[Lemma~6.3 and Appendix~C]{MP}. Temperedness supplies the
bounded coefficients required by \cite{GM}, after a
$T^\varepsilon$ normalization. Write $N_\pi(\sigma;T,T+H)$ for the
number, with multiplicity, of nontrivial zeros satisfying
$\Re\rho>\sigma$ and $T\le\Im\rho\le T+H$.

\begin{theorem}[A moment-to-density bound]\label{thm:general-density}
Let $\pi$ be a fixed globally tempered, unitary cuspidal representation
of $\GL_d(\A_F)$ over a fixed number field $F$; zeta is also allowed.
For fixed $0<a\le1$, $\ell\ge1$, and $\theta\ge a$, suppose
\begin{equation}\label{eq:moment-hypothesis}
 \int_U^{U+U^a}|L(\tfrac12+it,\pi)|^{2\ell}\dd t
 \ll_\eta U^{\theta+\eta}
\end{equation}
for every $\eta>0$, uniformly for all sufficiently large real $U$.
Then, for $3/5\le\sigma<1$ and $T\ge3$,
\begin{equation}\label{eq:general-density}
 \begin{gathered}
 N_\pi(\sigma;T,T+T^a)\ll T^{(1-\sigma)\mathcal D+\varepsilon},
 \\[2pt]
 \mathcal D=\max\!\left\{\frac{15a}{3+5\sigma},
       \frac{4\theta}{\ell(2\sigma-1)+4(1-\sigma)}\right\}.
 \end{gathered}
\end{equation}
The implied constants may also depend on $\pi,a,\ell,\theta$.
\end{theorem}

Combining this with the short-interval fourth moment of zeta due
to Heath-Brown \cite{HB} and Iwaniec \cite{Iw}, in the form recorded
in \cite[(2.1)]{YZ}, and the earlier density bounds gives:

\begin{theorem}[Zeta zeros in short strips]\label{thm:zeta-density}
For $2/3\le a\le1$, $1/2\le\sigma<1$, and $T\ge3$,
\begin{equation}\label{eq:zeta-density}
 N_\zeta(\sigma;T,T+T^a)\ll_{a,\varepsilon}T^{a\mathcal B(\sigma)+\varepsilon},
 \qquad
 \mathcal B(\sigma)=
 \begin{cases}
 \dfrac{3(1-\sigma)}{2-\sigma},&\dfrac12\le\sigma\le\dfrac7{10},\\[5pt]
 \dfrac{15(1-\sigma)}{3+5\sigma},&\dfrac7{10}\le\sigma\le\dfrac45,\\[5pt]
 \dfrac{3(1-\sigma)}{3\sigma-1},&\dfrac45\le\sigma\le\dfrac56,\\[5pt]
 2(1-\sigma),&\dfrac56\le\sigma<1.
 \end{cases}
\end{equation}
\end{theorem}

The improvement over \cite[Theorem~2.1]{YZ} is strict on
$7/10<\sigma<4/5$: at $\sigma=3/4$, the exponent decreases
from $3a/5$ to $5a/9$. The case $a=1$ recovers the corresponding
bound of Guth--Maynard; the gain for $a<1$ concerns localization
in the ordinate. With the moment inputs used here, this argument
does not improve the cusp-form density bounds, even in the
holomorphic case where Deligne \cite{Del} gives temperedness.
For general Maass forms, temperedness is itself unproved.

OpenAI's quasi-Riemann papers establish zero-free half-planes
$\Re s>11/12$ \cite{OAI11} and $\Re s>7/8$ \cite{OAI7}
for zeta and Dirichlet $L$-functions, as well as finite-order
Hecke $L$-functions over $\Q(\sqrt{-3})$. Liu
\cite[Theorem~27.1]{BLiu} further improves the boundary to
$0.874957019420098946\ldots$. We use $7/8$ in our main formulas
and record Liu's refinement in Section~\ref{sec:zeta}.
Thus $N_\zeta(\sigma;T,T+T^a)=0$ for $\sigma\ge7/8$ under
our strict counting convention. Possible zeros on $\Re s=7/8$
are retained in the weighted sums. The density improvement on
$(7/10,4/5)$ is unaffected: these zero-free results do not replace
the estimates there. They do not cover the general cusp forms
in Theorems~\ref{thm:gl2-sum}--\ref{thm:gl3-sum}.

Hypothesis S of Iwaniec, Luo, and Sarnak \cite[(1.62)]{ILS} predicts
\begin{equation}\label{eq:hypothesis-s}
 \sum_{\substack{p\le X\\p\equiv b\!\!\pmod c}}
 e\!\left(\frac{2\sqrt p}{c}\right)
 \ll_\varepsilon X^{1/2+\varepsilon},\qquad (b,c)=1,
\end{equation}
uniformly for $X\ge1$ and integers $c\ge1$, where $e(x)=e^{2\pi i x}$.
A bound $O_\varepsilon(c^A X^{\vartheta+\varepsilon})$ with fixed $A$
and $\vartheta<3/4$ would already enlarge the support in their
one-level density theorem. Vinogradov \cite[Ch.~4, \S2]{Vin} proved
$\sum_{p\le X}e(\alpha\sqrt p)\ll X^{7/8+\varepsilon}$ for fixed
$\alpha\ne0$. Ren \cite{Ren} obtained
\begin{equation}\label{eq:ren}
 \sum_{X<n\le2X}\Lambda(n)e(\alpha n^\beta)
 \ll_{\alpha,\beta}
 \bigl(X^{3/4+\beta/10}+X^{1-\beta/2}\bigr)(\log X)^{11},
 \qquad 0<\beta\le\tfrac12,
\end{equation}
giving exponent $4/5$ for the square-root phase. The classical
density hypothesis would give exponent $3/4$ in that case; see
\cite{ILS,Ren}. Further background is given in \cite{YeRIMS}.

We apply the short-strip estimates to short-interval analogues
for zeta and automorphic $L$-functions. Define $a_\pi(n)$ in a right half-plane by
\begin{equation}\label{eq:logderivative}
 -\frac{L'}{L}(s,\pi)=\sum_{n\ge1}a_\pi(n)\Lambda(n)n^{-s},
\end{equation}
and, for fixed $\phi\in C_c^\infty((-1,1))$, put
\begin{equation}\label{eq:sum}
 S_\pi(\alpha,\beta;K,L)
 =\sum_{n\ge1}a_\pi(n)\Lambda(n)e(\alpha n^\beta)
                 \phi\!\left(\frac{n-K}{L}\right).
\end{equation}
Thus $a_\zeta(n)=1$; at higher prime powers the cusp-form
coefficients in \eqref{eq:logderivative} need not equal the standard
Dirichlet coefficients. All forms, characters and
$\alpha\in\R\setminus\{0\}$ are fixed. Put
\begin{equation}\label{eq:scales}
 0<\beta\le1,\qquad L=K^\delta,\qquad
 1-\beta<\delta<1,\qquad \lambda=\beta+\delta-1>0.
\end{equation}
For the two cusp-form applications we impose the additional condition
$\delta>1-\beta/2$. Fix $R>0$ and write
\begin{equation}\label{eq:frequency-parameters}
 1\le m\le K^R,\qquad u=\frac{\log m}{\log K},\qquad
 b=\beta+u,\qquad v=\lambda+u.
\end{equation}
Then $0<v<b$, and $b/2<v<b$ in the cusp-form applications.
All bounds are uniform in $m$; constants may depend on $R$, the
fixed data and $\varepsilon$. A saving means
$O(K^{\delta-\omega+\varepsilon})$ with fixed $\omega>0$.

Define
\begin{equation}\label{eq:q-function}
 \mathcal Q(x)=
 \begin{cases}
 7/8+x/4,&0\le x\le91/230,\\[2pt]
 7/10+9x/13,&91/230\le x\le169/300,\\[2pt]
 2+3x-2\sqrt{3x},&169/300\le x\le3/4,\\[2pt]
 1/2+x,&3/4\le x<\infty.
 \end{cases}
\end{equation}

\begin{theorem}[The Riemann zeta function]\label{thm:zeta-sum}
Under \eqref{eq:scales}, fix $R>0$. Uniformly for integers
$1\le m\le K^R$, put $u=\log m/\log K$ and
\[
 \nu_m=\max\{\beta+\delta-1+u,\ 1-\delta,\ 2(\beta+u)/3\}.
\]
Then
\begin{equation}\label{eq:zeta-sum}
 \begin{split}
 S_\zeta(m\alpha,\beta;K,K^\delta)
 &\ll K^{\mathcal Q(\nu_m)+\varepsilon}
       \min\{K^{\delta-1},\,m^{-1/2}K^{-\beta/2}\}.
 \end{split}
\end{equation}
\end{theorem}

\begin{corollary}\label{cor:zeta}
Put $\beta_*=(78+24\sqrt3)/121$. For $0<\beta<1$, define
\begin{equation}\label{eq:zeta-threshold}
 d(\beta)=
 \begin{cases}
 \max\{17/30,1-\beta\},&0<\beta<13/20,\\[3pt]
 7/10-\beta/26,&13/20\le\beta\le169/200,\\[3pt]
 2+3\beta/2-2\sqrt{2\beta},&169/200\le\beta\le\beta_*,\\[3pt]
 2-5\beta/4-\dfrac{\sqrt3}{2}\sqrt{1-\beta},&\beta_*\le\beta<1.
 \end{cases}
\end{equation}
At $m=1$, estimate \eqref{eq:zeta-sum} gives a strict power saving
over $L$ exactly when $d(\beta)<\delta<1$.
In this range, uniformly for $1\le m\le K^R$,
\begin{equation}\label{eq:zeta-saving}
 \begin{aligned}
 S_\zeta(m\alpha,\beta;K,K^\delta)
 &\ll m^{1/2}K^{\delta-\omega(\beta,\delta)+\varepsilon},\\
 \omega(\beta,\delta)&=\delta+\max\{\beta/2,1-\delta\}\\
 &\quad-\mathcal Q\!\left(\max\{\beta+\delta-1,1-\delta,2\beta/3\}\right)>0.
 \end{aligned}
\end{equation}
Thus $m\le K^r$, with fixed $0\le r<2\omega$, preserves saving
$\omega-r/2>0$. The minimum threshold is $17/30$, for
$13/30\le\beta<13/20$.
\end{corollary}

For example, the square-root phase admits every $17/30<\delta<1$.
The exact frequency bound in Theorem~\ref{thm:zeta-sum} gives
\begin{equation}\label{eq:short-frequency-example}
 S_\zeta(m\alpha,1/2;K,K^{3/5})\ll K^{15/26+\varepsilon},
 \qquad 1\le m\le K^{1/10},
\end{equation}
with no frequency loss in this range.

The tables give exponents $E$ with
$S_\pi(m\alpha,\beta;K,K^\delta)\ll m^{1/2}K^{E+\varepsilon}$.
The tabulated saving $\omega=\delta-E$ is for $m=1$; it decreases
by at most $r/2$ for $m\le K^r$. Sections~4--6 prove the ranges.
\begin{table}[H]
\centering
\caption{The zeta function, using the $7/8$ zero-free boundary.}\label{tab:zeta}
\begin{tabular}{@{}cccc@{}}
\toprule
$\beta$ & $\delta$ & $E$ & $\omega$\\
\midrule
$1/2$ & $3/5$ & $15/26$ & $3/130$\\
$1/2$ & $2/3$ & $5/8$ & $1/24$\\
$1/2$ & $4/5$ & $17/24$ & $11/120$\\
$2/3$ & $17/25$ & $263/390$ & $11/1950$\\
$8/9$ & $7/10$ & $2/3$ & $1/30$\\
\bottomrule
\end{tabular}
\end{table}

For $b>0$ and $b/2<v<b$, put $C=b+3v$ and
\begin{equation}\label{eq:f2}
 \mathcal F_2(b,v)=
 \begin{cases}
 3/4+C/5,&C\le25/16,\\[3pt]
 2+C-2\sqrt C,&C\ge25/16,\quad(b+2v)^2\le C,\\[4pt]
 \displaystyle v+\frac{b+v}{b+2v},
                         &(b+2v)^2\ge C.
 \end{cases}
\end{equation}
The formulas agree at the common boundaries on this domain.
Since $b+2v<4C/5$, the first condition already implies
$(b+2v)^2<C$.

\begin{theorem}[$\GL(2)$ cusp forms]\label{thm:gl2-sum}
Let $g$ be a fixed self-dual holomorphic or Maass Hecke cusp form
for $\Gamma_0(N)$, and $\chi$ a fixed real primitive character
modulo $Q$, with $N\mid Q$. Assume \eqref{eq:scales} and $\delta>1-\beta/2$, and fix $R>0$.
Uniformly for integers $1\le m\le K^R$, with $u=\log m/\log K$,
\begin{equation}\label{eq:gl2-sum}
 S_{g\otimes\chi}(m\alpha,\beta;K,K^\delta)
 \ll m^{-1/2}K^{\max\{1,\mathcal F_2(\beta+u,\lambda+u)\}-\beta/2+\varepsilon}.
\end{equation}
\end{theorem}

\begin{corollary}\label{cor:gl2}
Put $\beta_c=(\sqrt{17}-3)/2$. For $0<\beta<2/3$, define
\begin{equation}\label{eq:gl2-threshold}
 d_2(\beta)=
 \begin{cases}
 1-\beta/2,&0<\beta\le1/2,\\[3pt]
 3/8+3\beta/4,&1/2\le\beta\le11/20,\\[3pt]
 2-7\beta/4-\tfrac12\sqrt{3-5\beta},&11/20\le\beta\le\beta_c,\\[3pt]
 \displaystyle1-\beta+\frac{\beta^2}{2(1-\beta)},
                                      &\beta_c\le\beta<2/3.
 \end{cases}
\end{equation}
At $m=1$, Theorem~\ref{thm:gl2-sum} gives a strict power saving over $L$
exactly when $0<\beta<2/3$ and $d_2(\beta)<\delta<1$.
In this range, uniformly for $1\le m\le K^R$,
\begin{equation}\label{eq:gl2-saving}
 \begin{aligned}
 S_{g\otimes\chi}(m\alpha,\beta;K,K^\delta)
 &\ll m^{1/2}K^{\delta-\omega_2(\beta,\delta)+\varepsilon},\\
 \omega_2(\beta,\delta)&=\delta+\beta/2
                  -\max\{1,\mathcal F_2(\beta,\lambda)\}>0.
 \end{aligned}
\end{equation}
Thus $1\le m\le K^r$, with fixed $0\le r<2\omega_2$, gives saving $\omega_2-r/2>0$.
The minimum of $d_2(\beta)$ is $3/4$, attained at $\beta=1/2$.
\end{corollary}

For example, when $\beta=1/2$, Theorem~\ref{thm:gl2-sum}
gives a power saving at $m=1$ for every $3/4<\delta<1$.
More explicitly, $S_{g\otimes\chi}(m\alpha,1/2;K,K^\delta)
\ll m^{1/2}K^{3/10+3\delta/5+\varepsilon}$ if $3/4<\delta\le4/5$.

\begin{table}[H]
\centering
\caption{Twists of $\GL(2)$ cusp forms.}\label{tab:gl2}
\begin{tabular}{@{}cccc@{}}
\toprule
$\beta$ & $\delta$ & $E$ & $\omega$\\
\midrule
$2/5$ & $17/20$ & $4/5$ & $1/20$\\
$1/2$ & $4/5$ & $39/50$ & $1/50$\\
$3/5$ & $9/10$ & $71/80$ & $1/80$\\
\bottomrule
\end{tabular}
\end{table}

\begin{theorem}[Self-dual $\GL(3)$ cusp forms]\label{thm:gl3-sum}
Let $f$ be a fixed self-dual Hecke--Maass cusp form for $\SL_3(\Z)$.
Assume \eqref{eq:scales} and $\delta>1-\beta/2$, and fix $R>0$. Uniformly for integers
$1\le m\le K^R$, with $u=\log m/\log K$,
\begin{equation}\label{eq:gl3-sum}
 \begin{aligned}
 S_f(m\alpha,\beta;K,K^\delta)
 &\ll m^{-1/2}K^{1-\beta/2+\varepsilon}\\
 &\quad +m^{1/2}K^{\delta+\beta/2-
               (\lambda+u)/(2(\beta+\lambda+2u))+\varepsilon}.
 \end{aligned}
\end{equation}
\end{theorem}

\begin{corollary}\label{cor:gl3}
For $0<\beta<1/2$, put
\begin{equation}\label{eq:gl3-threshold}
 d_3(\beta)=
 \begin{cases}
 1-\beta/2,&0<\beta\le1/3,\\[3pt]
 \displaystyle1-\beta+\frac{\beta^2}{1-\beta},&1/3\le\beta<1/2.
 \end{cases}
\end{equation}
At $m=1$, Theorem~\ref{thm:gl3-sum} gives a strict power saving over $L$
exactly when $0<\beta<1/2$ and $d_3(\beta)<\delta<1$.
In this range, uniformly for $1\le m\le K^R$,
\begin{equation}\label{eq:gl3-saving}
 \begin{aligned}
 S_f(m\alpha,\beta;K,K^\delta)
 &\ll m^{1/2}K^{\delta-\omega_3(\beta,\delta)+\varepsilon},\\
 \omega_3(\beta,\delta)&=\min\left\{
 \lambda-\frac\beta2,\ \frac{\lambda}{2(\beta+\lambda)}-\frac\beta2
                         \right\}>0.
 \end{aligned}
\end{equation}
Thus $1\le m\le K^r$, with fixed $0\le r<2\omega_3$, gives saving $\omega_3-r/2>0$.
The minimum of $d_3(\beta)$ is $5/6$, attained at $\beta=1/3$.
\end{corollary}

For example, when $\beta=1/3$, Theorem~\ref{thm:gl3-sum}
gives a power saving at $m=1$ for every $5/6<\delta<1$.
More explicitly,
\[
 S_f(m\alpha,1/3;K,K^\delta)\ll m^{1/2}K^{\delta-\omega_3+\varepsilon},
 \qquad \omega_3=\frac{\delta-5/6}{3\delta-1}>0.
\]

\begin{table}[H]
\centering
\caption{Self-dual $\GL(3)$ cusp forms.}\label{tab:gl3}
\begin{tabular}{@{}cccc@{}}
\toprule
$\beta$ & $\delta$ & $E$ & $\omega$\\
\midrule
$1/4$ & $9/10$ & $7/8$ & $1/40$\\
$1/3$ & $8/9$ & $77/90$ & $1/30$\\
$2/5$ & $9/10$ & $31/35$ & $1/70$\\
\bottomrule
\end{tabular}
\end{table}

Neither Theorem~\ref{thm:gl2-sum} nor Theorem~\ref{thm:gl3-sum}
assumes Ramanujan. These savings are relative to $L$ itself.
The explicit formula of Rudnick--Sarnak \cite{RS} reduces the
problem to zeros at height $\asymp K^b$. For zeta we retain
both the stationary spread $K^v$ and the Fourier width $K^{1-\delta}$;
for the cusp forms the former dominates. We use Ye--Zhang's short-strip density
bounds \cite[Theorems~2.2--2.3]{YZ} for the cusp forms; their
$\GL(2)$ input uses the moments of Ye \cite{YeMoment} and
Lau--Liu--Ye \cite{LLY}.

Related sums over all integers, with normalized Fourier coefficients
and without $\Lambda(n)$, exhibit resonance: square-root phases
can produce $X^{3/4}$ main terms for fixed $\GL(2)$ forms
\cite{ILS,RY10,SW}. Ren--Ye developed the $\GL(3)$ and higher-rank
theory \cite{RY14,RY15GL3,RY15GLm,RY16}; Ye \cite{Ye22} studied
bounds toward Hypothesis S as the form varies. Spectral moments
and related bounds were obtained by Salazar--Ye \cite{SY} and
Gillespie--Samanta--Ye \cite{GSY}, while Gillespie \cite{Gillespie}
treated quadratic base change. See also McKee--Sun--Ye's weighted
stationary phase and its applications \cite{MSY,MSYsub}, and the
survey \cite{YeRIMS}.

For ordinary primes, these consequences should be distinguished
from earlier direct exponential-sum estimates. Taylor approximation
and the polynomial-phase estimate of Matom\"aki, Shao, Tao,
and Ter\"av\"ainen \cite[Corollary~1.3(ii)]{MSTT} imply equidistribution of
$\{\alpha p^\beta\}$ on every interval $(X,X+X^\delta]$ when
$0<\beta<1$ and $\delta>\max\{5/8,1-\beta\}$.
The quantitative Type I and II estimates of Matom\"aki--Shao
\cite[Propositions~2.1--2.2]{MS} also yield some power saving
for $\delta>\max\{2/3,1-\beta/2\}$. Our frequency bounds also imply
such equidistribution, with a power-saving discrepancy in the range
of Corollary~\ref{cor:zeta}, by Erd\H{o}s--Tur\'an and smooth approximation.
In particular, the present bounds apply for $0<\beta<13/20$ and
$\delta>\max\{17/30,1-\beta\}$, using the new zero-free half-plane
in addition to localized density. The quoted direct estimates and
our bounds have different inputs; we focus on the transfer from
short-strip density to oscillatory prime sums. No assertion of an
optimal discrepancy exponent is made here.

Section~\ref{sec:density} proves the zero-density theorems.
Section~\ref{sec:explicit} establishes the explicit formula and the
transform estimates.
Sections~4, 5, and~6 apply these results to zeta, cusp forms for
$\Gamma_0(N)$, and Maass forms for $\SL_3(\Z)$, respectively.

\section{Zero density on short strips}\label{sec:density}

We prove Theorems~\ref{thm:general-density} and \ref{thm:zeta-density}.
Throughout the proof of the former, $H=T^a$, $3/5\le\sigma<1$,
and all $L$-function data are fixed. Choose $X=T^\kappa$, with
$\kappa>0$ sufficiently small in terms of the prescribed final
$\varepsilon$. The detector cutoff $Y$ will satisfy
$T^{c_0}\le Y\le T^{C_0}$ for fixed $C_0\ge c_0>\kappa$ and
will be chosen at the end.

Let $A_\pi(n)$ and $\mu_\pi(n)$ be the Dirichlet coefficients of
$L(s,\pi)$ and $L(s,\pi)^{-1}$, respectively, grouped by norm if
$F\ne\Q$. Put
\begin{equation}\label{eq:detector-coefficients}
 M_X(s)=\sum_{n\le X}\mu_\pi(n)n^{-s},\qquad
 b_X(n)=\sum_{\substack{d\mid n\\d\le X}}
                      \mu_\pi(d)A_\pi(n/d).
\end{equation}
Temperedness at finite places gives
$|A_\pi(n)|+|\mu_\pi(n)|\ll_\nu n^\nu$ for every $\nu>0$.
Thus $b_X(1)=1$, $b_X(n)=0$ for $2\le n\le X$, and
$b_X(n)\ll_\nu n^\nu$, uniformly in $X$. The same bound holds
for each fixed convolution power of truncated coefficients
$b_X(n)e^{-n/Y}$: a fixed divisor function majorizes the coefficients.
After $\kappa$ and the powering indices have been fixed, we choose
the coefficient exponents and all further small losses sufficiently
small to absorb their total cost into $T^\varepsilon$.

For a zero $\rho=u+it$, $u>\sigma$ and $T\le t\le T+H$,
Mellin inversion and a contour shift give
\begin{align}
 e^{-1/Y}+\sum_{n>X}b_X(n)n^{-\rho}e^{-n/Y}
       &=J(\rho)+O(T^{-100}),                         \label{eq:detector}\\
 J(\rho)&=\frac1{2\pi}\int_\R
 L(\tfrac12+i(t+v),\pi)M_X(\tfrac12+i(t+v))
 \notag\\[-2pt]
 &\hspace{4em}\times
 \Gamma(\tfrac12-u+iv)Y^{1/2-u+iv}\dd v.           \label{eq:detector-integral}
\end{align}
This is the detector used in \cite[Sections~4--5]{YZ}.
The residue at the origin vanishes because $L(\rho,\pi)=0$;
any fixed pole of $L(s,\pi)$ contributes an exponentially small
term by Stirling's formula. We may truncate the sum at
$Y(\log T)^2$ and the integral at $|v|\le V:=(\log T)^2$,
with error $O(T^{-100})$. Every such zero therefore satisfies
at least one of
\begin{equation}\label{eq:zero-types}
 \text{Type I: }\quad
 \left|\sum_{X<n\le Y(\log T)^2}b_X(n)n^{-\rho}e^{-n/Y}\right|
       \ge\tfrac14,
 \qquad
 \text{Type II: }\quad |J(\rho)|\ge\tfrac14.
\end{equation}
Let $R_{\mathrm I}$ and $R_{\mathrm{II}}$ count these zeros, with
multiplicity. Then $N_\pi(\sigma;T,T+H)\le R_{\mathrm I}+R_{\mathrm{II}}$.
We use the standard local bound
\begin{equation}\label{eq:density-local-count}
 \#\{\rho:v\le\Im\rho\le v+1\}\ll_\pi\log T
 \qquad(v\asymp T).
\end{equation}

\begin{lemma}[Type II zeros]\label{lem:type-II}
Under the moment hypothesis \eqref{eq:moment-hypothesis},
\begin{equation}\label{eq:type-II-bound}
 R_{\mathrm{II}}\ll T^{\theta+\varepsilon}
                          Y^{-\ell(2\sigma-1)}.
\end{equation}
\end{lemma}

\begin{proof}
We localize the integral argument of
\cite[Lemma~6.3 and Appendix~C]{MP}. The gamma factor in
\eqref{eq:detector-integral} is uniformly bounded for
$3/5\le u\le1$ and decays exponentially in $|v|$. Hence a Type II
zero satisfies
\[
 Y^{\sigma-1/2}\ll\int_{-V}^{V}
       |L(\tfrac12+i(t+v),\pi)M_X(\tfrac12+i(t+v))|\dd v.
\]
Raise this to the power $2\ell$ and apply H\"older.
By \eqref{eq:density-local-count}, the intervals $[t-V,t+V]$
overlap at most $O(V\log T)$ times, including multiplicities.
Summing therefore gives
\begin{equation}\label{eq:type-II-raw}
 R_{\mathrm{II}}Y^{\ell(2\sigma-1)}
 \ll(\log T)^{O_\ell(1)}
 \int_{T-V}^{T+H+V}|L(\tfrac12+iv,\pi)M_X(\tfrac12+iv)|^{2\ell}\dd v.
\end{equation}
Since $V=o(H)$, this interval is covered by $O(1)$ intervals
$[U,U+U^a]$, $U\asymp T$. The uniform moment hypothesis and
$M_X(1/2+iv)\ll_\nu X^{1/2+\nu}$ bound the right side by
$T^{\theta+\eta}X^{\ell+2\ell\nu}(\log T)^{O_\ell(1)}$.
Choose $\kappa$, $\eta$, and $\nu$ sufficiently small to obtain
\eqref{eq:type-II-bound}. This uses only the assumed untwisted
moment on the enlarged short interval.
\end{proof}

\begin{lemma}[Type I zeros]\label{lem:type-I}
Without a moment hypothesis,
\begin{equation}\label{eq:type-I-bound}
 R_{\mathrm I}\ll T^\varepsilon
 \left(H^{15(1-\sigma)/(3+5\sigma)}+Y^{4(1-\sigma)}\right).
\end{equation}
\end{lemma}

\begin{proof}
Dyadic subdivision assigns each Type I zero to a polynomial
$D_M(s)=\sum_{M<n\le2M}d_n n^{-s}$, where
$d_n=b_X(n)e^{-n/Y}$ on the detector range and is zero otherwise,
$X/2\le M\le Y(\log T)^2$, and
$|D_M(u+it)|\gg1/\log T$. There are $O(\log T)$ choices of $M$.
We first move all evaluations to the line $\Re s=\sigma$, as in
\cite[Section~13.1]{GM}, keeping track of their separation.
For $q=u-\sigma\in[0,2/5]$, take a fixed smooth compactly
supported cutoff $\chi_0=1$ on $[0,\log2]$, and put
$f_q(v)=\chi_0(v)e^{-qv}$. With
$\widehat f_q(\xi)=\int f_q(v)e^{-i\xi v}\dd v$, Fourier inversion gives
\begin{equation}\label{eq:move-real-part}
 D_M(u+it)=\frac{M^{-q}}{2\pi}
 \int_\R\widehat f_q(\xi)M^{-i\xi}
                D_M(\sigma+i(t-\xi))\dd\xi.
\end{equation}
The transforms have uniformly bounded $L^1$ norms and rapid decay.
For any sufficiently small fixed $\eta>0$, truncation to
$|\xi|\le T^\eta$ has error $O(T^{-100})$. Since $M^{-q}\le1$,
there is an evaluation point $t_\rho$ with
$|t_\rho-t|\le T^\eta$ and
$|D_M(\sigma+it_\rho)|\gg1/\log T$.
These points need not be zero ordinates. By
\eqref{eq:density-local-count}, each interval of length $2$
contains at most $O(T^\eta\log T)$ of them, counted with
multiplicity. Greedy selection recovers a $1$-separated set at
that cost. Taking $\eta<a/2$, all selected points lie in an
interval of length $O(H)$. Translation and reflection of this
interval preserve coefficient absolute values, so the height
parameter in the large-value estimates is $O(H)$.

Introduce an independent length $Z\ge Y^2$, bounded by a fixed
power of $T$. For each $M$, choose a bounded positive integer $k$
such that
\begin{equation}\label{eq:powering-window}
 Z^{2/3}\le M^k\le Z(\log T)^4.
\end{equation}
If $M\le Z^{1/3}$, take the least $k$ with $M^k\ge Z^{2/3}$;
otherwise take $k=2$, using $M\le Y(\log T)^2$.
The bound $M\gg T^\kappa$ makes $k$ bounded independently of $T$.
Raise $D_M$ to this power, split its support dyadically, and group
the points by $M,k$, and the resulting block. If $e_n$ denotes
its convolution coefficients, multiply the block by $N^\sigma$
and divide by $C=\max(1,\max_n|e_n|)\le T^{\eta/2}$.
Thus $c_n=C^{-1}(N/n)^\sigma e_n$ defines a single polynomial
$P_N(v)=\sum_{N<n\le2N}c_n n^{-iv}$ for the group, with
\begin{equation}\label{eq:normalized-polynomial}
 |c_n|\le1,\qquad |P_N(t_\rho)|\ge N^\sigma T^{-\eta},\qquad
 Z^{2/3}\ll N\ll Z(\log T)^4.
\end{equation}
Here $\eta$ may be made arbitrarily small after $k$ is fixed.
The number of groups and all dyadic losses are logarithmic;
all extraction and coefficient-normalization costs are absorbed
into $T^\varepsilon$ below.

For a group of $R$ separated points, Guth--Maynard
\cite[Theorem~1.1]{GM} and the mean-value estimate
\cite[Lemma~3.1]{YZ}, respectively, give
\begin{align}
 R&\ll T^\varepsilon\bigl(N^{2(1-\sigma)}
              +N^{18/5-4\sigma}+HN^{12/5-4\sigma}\bigr),
                                                    \label{eq:GM-short}\\
 R&\ll T^\varepsilon\bigl(N^{2(1-\sigma)}+HN^{1-2\sigma}\bigr).
                                                    \label{eq:MV-short}
\end{align}
For $N\ge H$, \eqref{eq:MV-short} alone is bounded by its first
term. Put $q_\sigma=8/5-8\sigma/3\le0$. We claim that
\begin{equation}\label{eq:Z-bound}
 R_{\mathrm I}\ll T^\varepsilon
                   \bigl(Z^{2(1-\sigma)}+HZ^{q_\sigma}\bigr).
\end{equation}
For $\sigma\ge4/5$, the middle term of \eqref{eq:GM-short} is
at most the first, and \eqref{eq:normalized-polynomial} bounds
the first and last terms by the two terms of \eqref{eq:Z-bound}.
For $3/5\le\sigma\le4/5$, set
$v_\sigma=5(1-\sigma)/(9-10\sigma)\in[2/3,1]$.
Use \eqref{eq:GM-short} when $N\le Z^{v_\sigma}$; its middle
term is then at most $Z^{2(1-\sigma)}$.
When $N>Z^{v_\sigma}$, use \eqref{eq:MV-short} and the identity
\begin{equation}\label{eq:exponent-comparison}
 q_\sigma-v_\sigma(1-2\sigma)
 =\frac{250(\sigma-3/4)^2+3/8}{15(9-10\sigma)}>0.
\end{equation}
This proves \eqref{eq:Z-bound}, with logarithmic endpoint factors
absorbed into $T^\varepsilon$. Finally choose
\[
 Z=\max\{H^{15/(2(3+5\sigma))},Y^2\}.
\]
Since $2(1-\sigma)-q_\sigma=2(3+5\sigma)/15$,
we have $HZ^{q_\sigma}\le Z^{2(1-\sigma)}$, and
\eqref{eq:type-I-bound} follows.
\end{proof}

\begin{proof}[Proof of Theorem~\ref{thm:general-density}]
Set $b=\ell(2\sigma-1)$ and $c=4(1-\sigma)$.
Lemmas~\ref{lem:type-II} and \ref{lem:type-I} give
\begin{equation}\label{eq:combine-types}
 N_\pi(\sigma;T,T+T^a)
 \ll T^\varepsilon\left(
 T^{15a(1-\sigma)/(3+5\sigma)}+Y^c+T^\theta Y^{-b}\right).
\end{equation}
Take $Y=T^{\theta/(b+c)}$, so the last two terms are both
$T^{c\theta/(b+c)}$. Here $Y$ lies between fixed positive powers
of $T$, uniformly for $3/5\le\sigma<1$, so $X<Y$ after choosing
$\kappa$ sufficiently small. Allocating the earlier small losses
within the prescribed $\varepsilon$ proves
\eqref{eq:general-density}.
\end{proof}

\begin{proof}[Proof of Theorem~\ref{thm:zeta-density}]
The fourth-moment estimate of Heath-Brown \cite{HB} and Iwaniec
\cite{Iw}, as recorded in \cite[(2.1)]{YZ}, is
\begin{equation}\label{eq:zeta-fourth-moment}
 \int_T^{T+T^a}|\zeta(\tfrac12+it)|^4\dd t
          \ll_{a,\eta}T^{a+\eta},\qquad 2/3\le a\le1.
\end{equation}
For $a=1$, this also follows by covering $[T,2T]$ with
$O(T^{1/3})$ intervals of length comparable to $T^{2/3}$.
Theorem~\ref{thm:general-density}, with $\ell=2$ and $\theta=a$,
therefore gives, for $3/5\le\sigma<1$,
\begin{equation}\label{eq:zeta-direct-density}
 N_\zeta(\sigma;T,T+T^a)
 \ll T^{a(1-\sigma)\max\{15/(3+5\sigma),\,2\}+\varepsilon}.
\end{equation}
Ye--Zhang \cite[(2.2)--(2.4)]{YZ} give exponents
$3a(1-\sigma)/(2-\sigma)$ for $1/2\le\sigma\le3/4$,
$3a(1-\sigma)/(3\sigma-1)$ for $3/4<\sigma\le5/6$, and
$2a(1-\sigma)$ for $5/6<\sigma<1$.
For $a=1$, the same bounds follow from their Theorems~1.1 and~1.3
and \eqref{eq:zeta-fourth-moment}.
The comparisons
\[
 \frac{15}{3+5\sigma}\le\frac3{2-\sigma}
       \iff \sigma\ge\frac7{10},\qquad
 \frac{15}{3+5\sigma}\le\frac3{3\sigma-1}
       \iff \sigma\le\frac45
\]
show that taking the best of these estimates gives exactly
\eqref{eq:zeta-density}. The Type II entry $2$ does not govern
\eqref{eq:zeta-direct-density} on $[7/10,4/5]$; the four branches
agree at their common endpoints. The improvement is strict
precisely on $7/10<\sigma<4/5$.
\end{proof}

\section{The explicit formula}\label{sec:explicit}

Let $\pi$ denote zeta or one of the fixed cusp forms in
Theorems~\ref{thm:gl2-sum}--\ref{thm:gl3-sum}. Assume
\eqref{eq:scales} and take $K$ large enough that $L\le K/2$.
Set
\begin{equation}\label{eq:transform}
 \begin{aligned}
 w(x)&=e(\alpha x^\beta)\phi\!\left(\frac{x-K}{L}\right),
 &G(u)&=e^{u/2}w(e^u),\\
 h(z)&=\int_\R G(u)e^{izu}\dd u
      =\int_0^\infty w(x)x^{iz-1/2}\dd x.
 \end{aligned}
\end{equation}
Then $G\in C_c^\infty((0,\infty))$, $G(0)=0$, and $h$ is entire.
For a nontrivial zero write
\begin{equation}\label{eq:zero-notation}
 \rho=\sigma+it=\tfrac12+i\gamma,\qquad
 \gamma=t+i(1/2-\sigma),\qquad
 h(\gamma)=\int_0^\infty w(x)x^{\sigma-1+it}\dd x.
\end{equation}
No Riemann hypothesis is assumed.

Write the archimedean factor as
$\prod_{j=1}^{d_\pi}\Gamma_\R(s+\mu_j)$, where
$\Gamma_\R(s)=\pi^{-s/2}\Gamma(s/2)$ and $\Re\mu_j>-1/2$.
Put $D=\Gamma_\R'/\Gamma_\R$ and let $Q_\pi$ be the conductor.
Rudnick--Sarnak \cite[Proposition~2.1]{RS} gives
\begin{equation}\label{eq:explicit}
 S_\pi=-\sum_\rho h\!\left(\frac{\rho-1/2}{i}\right)
       +r_\pi\{h(-i/2)+h(i/2)\}+\mathcal A_\pi(G),
\end{equation}
where $r_\zeta=1$, $r_\pi=0$ for cusp forms, and
\begin{equation}\label{eq:arch-definition}
 \mathcal A_\pi(G)=\frac1{2\pi}\int_\R h(v)
 \left[\log Q_\pi+\sum_{j=1}^{d_\pi}
 \left\{D(\tfrac12+\mu_j+iv)
       +D(\tfrac12+\overline{\mu_j}-iv)\right\}\right]\dd v.
\end{equation}
The dual prime sum vanishes since $G(-\log n)=0$; the conductor
term equals $(\log Q_\pi)G(0)=0$. Zeros are counted with multiplicity.

The oscillation parameter below is
\begin{equation}\label{eq:tau}
 \tau=|\alpha|\beta LK^{\beta-1}=|\alpha|\beta K^\lambda;
\end{equation}
$\lambda$ continues to denote only the exponent in \eqref{eq:scales}.

\begin{proposition}[Pole and gamma terms]\label{prop:harmless}
For every $A>0$, with $\mu_*:=\min_j\Re\mu_j$,
\begin{align}
 |h(-i/2)|+K|h(i/2)|&\ll_A L(1+\tau)^{-A},
                                                \label{eq:pole-decay}\\
 \mathcal A_\pi(G)
 &=-\sum_{j=1}^{d_\pi}\int_{K-L}^{K+L}
          w(x)\frac{x^{-1-\mu_j}}{1-x^{-2}}\dd x
 \notag\\[-2pt]
 &\ll_A LK^{-1-\mu_*}(1+\tau)^{-A}.              \label{eq:gamma-decay}
\end{align}
Consequently, for every $B>0$,
\begin{equation}\label{eq:zero-reduction}
 S_\pi=-\sum_\rho h\!\left(\frac{\rho-1/2}{i}\right)+O_B(K^{-B}).
\end{equation}
\end{proposition}

\begin{proof}
Put $F(z)=\alpha(K+Lz)^\beta$ on $[-1,1]$. Its first derivative
is nonzero, its first two derivatives do not change sign, and
\[
 \min|F'|\asymp\tau,\qquad
 |F^{(j)}(z)|\ll_j\tau(L/K)^{j-1}\quad(j\ge2).
\]
For an amplitude $a\in C_c^\infty((-1,1))$ with
$a^{(j)}\ll_j U$, apply McKee--Sun--Ye
\cite[Theorem~1.1]{MSY}, for any integer $n\ge2$, with their
parameters $M=N=2$, $T=\tau$.
All boundary terms vanish, and their error is
$\ll_n U\sum_{j=0}^n\tau^j/(\min|F'|)^{n+j+1}
\ll_n U\tau^{-n-1}$. The trivial bound for $\tau<1$ therefore gives
\begin{equation}\label{eq:first-derivative}
 \int_{-1}^1 a(z)e(F(z))\dd z\ll_A U(1+\tau)^{-A}.
\end{equation}
Complex amplitudes are handled by their real and imaginary parts;
for $\beta=1$, $F''=0$ is allowed. Substituting $x=K+Lz$ in
$h(-i/2)=\int w(x)\dd x$ and $h(i/2)=\int w(x)x^{-1}\dd x$,
use $a=\phi$, $U=1$, and $a=\phi/(K+Lz)$, $U=K^{-1}$,
respectively. This proves \eqref{eq:pole-decay} directly for the
complete pole integrals.

For the gamma terms, the digamma representation
\cite[5.9.16]{DLMF}, together with Fourier inversion, gives for
$\Re a>0$
\begin{equation}\label{eq:gamma-transform}
 \frac1{2\pi}\int_\R h(v)D(a\pm iv)\dd v
 =-\int_0^\infty\frac{G(\pm u)e^{-au}}{1-e^{-2u}}\dd u.
\end{equation}
Indeed, writing $\psi=\Gamma'/\Gamma$ and $\gamma_E$ for Euler's constant,
insert $D(s)=-\tfrac12\log\pi+\tfrac12\psi(s/2)$ and
\[
 \psi(z)=-\gamma_E+
      \int_0^\infty\frac{e^{-v}-e^{-zv}}{1-e^{-v}}\dd v.
\]
The constant terms vanish by $G(0)=0$, and $v=2u$ yields
\eqref{eq:gamma-transform}. The interchange is justified before
splitting the numerator: near $v=0$ its quotient is $O_a(1+|t|)$
when $z=(a\pm it)/2$, and for large $v$ it decays exponentially;
$h(t)$ is rapidly decreasing. Since $G(-u)=0$, the term
$D(1/2+\overline{\mu_j}-iv)$ contributes zero.
For the other term take $a=1/2+\mu_j$ and substitute $x=e^u$.
This proves the identity in \eqref{eq:gamma-decay}. Its rescaled
amplitudes and all their fixed-order derivatives are
$O(K^{-1-\Re\mu_j})$, so \eqref{eq:first-derivative} proves its bound.
Finally, $\tau\asymp K^\lambda$ with $\lambda>0$; increasing $A$
absorbs the fixed powers of $K$ and proves \eqref{eq:zero-reduction}.
\end{proof}

We next localize the zero sum, allowing the curvature parameter to
be small. Define
\begin{equation}\label{eq:stationary-interval}
 \begin{gathered}
 I=\{-2\pi\alpha\beta x^\beta:K-L\le x\le K+L\},\\
 r=L/K,\qquad P=K^\beta r^2.
 \end{gathered}
\end{equation}
The absolute height of $I$ is $\asymp K^\beta$ and its length is
$\asymp K^\lambda$. Neither transform estimate below requires $P\to\infty$.

\begin{proposition}[Transform bounds]\label{prop:localization}
Uniformly for $0\le\sigma\le1$ and real $t$,
\begin{equation}\label{eq:second-derivative}
 |h(t+i(1/2-\sigma))|
 \ll\min\{LK^{\sigma-1},K^{\sigma-\beta/2}\}.
\end{equation}
For $t\notin I$ and every $A>0$,
\begin{equation}\label{eq:off-window}
 |h(t+i(1/2-\sigma))|
 \ll_A LK^{\sigma-1}
       (1+P+r\operatorname{dist}(t,I))^{-A}.
\end{equation}
\end{proposition}

\begin{proof}
The first entry of \eqref{eq:second-derivative} is the trivial bound.
In the variable $u=\log x$, the phase is $2\pi\alpha e^{\beta u}+tu$,
whose second derivative has size $\asymp K^\beta$.
The amplitude $e^{\sigma u}\phi((e^u-K)/L)$ has supremum norm
and total variation $O(K^\sigma)$. The weighted second derivative
test gives the second entry.

For the off-window estimate, set $u=\log K+rz$. Then
\[
 h(t+i(1/2-\sigma))
 =rK^\sigma e^{it\log K}\int a_{\sigma,r}(z)e(f_t(z))\dd z,
\]
where
\[
 a_{\sigma,r}(z)=e^{\sigma rz}\phi\!\left(\frac{e^{rz}-1}{r}\right),
 \qquad f_t(z)=\alpha K^\beta e^{\beta rz}+\frac{trz}{2\pi}.
\]
Choose fixed $0<c<c'<1$ with $\supp\phi\subset[-c,c]$, and integrate
over $[r^{-1}\log(1-c'r),r^{-1}\log(1+c'r)]$.
The amplitude vanishes near the endpoints and has uniformly bounded
derivatives. For $t\notin I$, with $d=\operatorname{dist}(t,I)$,
\[
 \min|f_t'|\gg P+rd,\qquad
 |f_t^{(j)}|\ll_j Pr^{j-2}\quad(j\ge2).
\]
The first inequality uses the fixed support gap. When $P+rd\ge1$,
repeated integration by parts gives arbitrary powers of $(P+rd)^{-1}$,
since the higher derivatives are $O_j(P+rd)$; there are no boundary
terms. When $P+rd<1$, the trivial bound suffices.
This proves \eqref{eq:off-window}, including the small-curvature case.
\end{proof}

\begin{proposition}[Uniform frequency reduction]\label{prop:frequency-reduction}\label{prop:wide-frequency}
Let $h_m,I_m$ be obtained by replacing $\alpha$ by $m\alpha$.
Under \eqref{eq:scales}--\eqref{eq:frequency-parameters}, put
\[
 \begin{gathered}
 w_0=1-\delta,\qquad a_0=\max\{b/2,w_0\},\\
 q=\max\{b-w_0,w_0\}.
 \end{gathered}
\]
Fix $0<\eta<\frac14\min\{w_0,\beta-w_0\}$, and enlarge $I_m$
by $K^{w_0+\eta}$ at each endpoint to obtain $J_m$.
It has absolute height $\asymp K^b$ and length $O(K^{q+\eta})$.
For every $B>0$, uniformly for $1\le m\le K^R$,
\begin{equation}\label{eq:wide-frequency}
 |S_\pi(m\alpha,\beta;K,K^\delta)|
 \ll K^{-a_0}
 \sum_{\substack{\rho=\sigma+it\\t\in J_m,\ \sigma\ge1/2}}K^\sigma
       +K^{-B}.
\end{equation}
If also $\delta>1-\beta/2$, the original stationary window suffices:
\begin{equation}\label{eq:frequency-reduction}
 |S_\pi(m\alpha,\beta;K,K^\delta)|
 \ll K^{-b/2}
 \sum_{\substack{\rho=\sigma+it\\t\in I_m,\ \sigma\ge1/2}}K^\sigma
       +K^{-B},
\end{equation}
where $|I_m|\asymp K^v$. Constants may depend on $\eta$ and $B$.
\end{proposition}

\begin{proof}
The transform bounds above replace $P$ by $P_m=mP$ and yield
$|h_m((\rho-1/2)/i)|\ll K^{\sigma-a_0}$.
For $t\notin J_m$, the off-window bound has
$r\operatorname{dist}(t,I_m)\ge K^\eta$.
Using the unit-interval zero count
$\#\{\rho:t\le\Im\rho<t+1\}\ll_\pi\log(3+|t|)$ and comparing
with an integral gives, for every integer $N\ge3$,
\[
 \sum_{\Im\rho\notin J_m}|h_m((\rho-1/2)/i)|
 \ll_{N,R}K^{1-\eta(N-1)}\log K.
\]
Choosing $N$ sufficiently large makes this $O(K^{-B})$.
The choice of $\eta$ gives $q+\eta<b$, so $J_m$ remains at height
$\asymp K^b$. Proposition~\ref{prop:harmless} applies uniformly
because $\tau_m\asymp K^{b-w_0}\to\infty$. Apply the explicit
formula and use $\rho\mapsto1-\overline\rho$ to replace zeros
with $\sigma<1/2$ by their partners. This proves \eqref{eq:wide-frequency}.

If $\delta>1-\beta/2$, then $b/2>w_0$ and
$P_m\ge K^{\beta-2w_0}$ grows as a fixed power.
The same summation outside $I_m$ gives
\[
 \sum_{\Im\rho\notin I_m}|h_m((\rho-1/2)/i)|
 \ll_{N,R}K(1+P_m)^{1-N}\log K=O(K^{-B}),
\]
after increasing $N$. The second derivative bound now proves
\eqref{eq:frequency-reduction}.
\end{proof}

When a density estimate is applied with a varying strip
exponent $a$, uniformity follows by a finite mesh: enlarge $a$ to
the next mesh point and use a bounded covering. All relevant $a$
lie in a fixed compact subinterval of $(1/2,1)$, and their density
exponents change by $O_R(\eta)$ on a mesh of width $\eta$.
Choose the mesh and the density loss sufficiently small in terms
of $\varepsilon/(1+\beta+R)$. This absorbs all losses into
$K^\varepsilon$ without assuming uniformity of an implicit constant
in a continuously varying parameter.

\section{The Riemann zeta function}\label{sec:zeta}

\begin{proof}[Proof of Theorem~\ref{thm:zeta-sum}]
Put $w_0=1-\delta$, $a_0=\max\{b/2,w_0\}$ and
$q=\max\{b-w_0,w_0\}$ as in Proposition~\ref{prop:wide-frequency}.
The zero window $J_m$ has absolute height $\asymp K^b$ and
length $O(K^{q+\eta})$. Write $\nu=\max\{q,2b/3\}=\nu_m$.
After reflection if the height is negative, cover $J_m$ by a
bounded number of intervals $[T,T+T^a]$, with $T\asymp K^b$
and $a=(\nu+2\eta)/b\in[2/3,1)$. Choose $\eta$ small enough
in terms of $\varepsilon$ and $R$. Theorem~\ref{thm:zeta-density}
and the finite-mesh argument in Section~\ref{sec:explicit} give
\begin{equation}\label{eq:zeta-density-window}
 N_\zeta(s;J_m)\ll K^{\nu\mathcal B(s)+\varepsilon},
 \qquad 1/2\le s<7/8.
\end{equation}
All losses from the enlargement and the mesh are absorbed into
$K^\varepsilon$. By \cite{OAI7}, every nontrivial zeta zero
has real part at most $7/8$. Put
$M_J=\#\{\rho:\Re\rho\ge1/2,\ \Im\rho\in J_m\}$.
Partial summation, including multiplicities, gives the exact identity
\begin{equation}\label{eq:zeta-partial-summation}
 \sum_{\substack{\rho=\sigma+it\\t\in J_m,\ \sigma\ge1/2}}K^\sigma
 =K^{1/2}M_J+(\log K)\int_{1/2}^{7/8}K^sN_\zeta(s;J_m)\dd s.
\end{equation}
This includes zeros on both boundary lines. Since
$M_J\ll K^{q+\eta+\varepsilon}$, $q\le\nu$ and
$\mathcal B(1/2)=1$, it remains to verify
\begin{equation}\label{eq:zeta-maximization}
 \max_{1/2\le s\le7/8}\{s+x\mathcal B(s)\}=\mathcal Q(x),
 \qquad x\ge0.
\end{equation}

The function is affine on $[5/6,7/8]$ and convex on each of
$[4/5,5/6]$ and $[7/10,4/5]$. On $[1/2,7/10]$, put
\[
 G_x(s)=s+\frac{3x(1-s)}{2-s},\qquad
 G_x'(s)=1-\frac{3x}{(2-s)^2}.
\]
The possible maxima are therefore
\begin{equation}\label{eq:zeta-four-bands}
 \max\left\{\frac78+\frac x4,\ \frac56+\frac x3,
 \ \frac45+\frac{3x}7,\ \max_{1/2\le s\le7/10}G_x(s)\right\}.
\end{equation}
The middle two entries are dominated by the maximum of the
first and $G_x(7/10)=7/10+9x/13$: they are below the first
up to $x=91/230$, and below $G_x(7/10)$ thereafter.
The concave function $G_x$ has its maximum at $7/10$ for
$x\le169/300$, at $2-\sqrt{3x}$ for
$169/300\le x\le3/4$, and at $1/2$ for $x\ge3/4$.
The first two relevant lines cross at
$7/8+x/4=7/10+9x/13$, or $x=91/230$.
This proves \eqref{eq:zeta-maximization}. Combining it with
\eqref{eq:wide-frequency} proves the theorem, since
$K^{-a_0}=\min\{K^{\delta-1},m^{-1/2}K^{-\beta/2}\}$.
\end{proof}

\begin{proof}[Proof of Corollary~\ref{cor:zeta}]
First set $m=1$. If $\delta\le1-\beta/2$, then
$a_0=1-\delta$ and $\nu=\max\{1-\delta,2\beta/3\}$.
The saving condition is exactly $\mathcal Q(\nu)<1$.
By \eqref{eq:q-function}, this is equivalent to $\nu<13/30$,
or
\begin{equation}\label{eq:short-saving-range}
 \beta<13/20,\qquad \delta>17/30,
\end{equation}
in addition to the standing condition $\delta>1-\beta$.

For $\delta\ge1-\beta/2$, put $\nu=\max\{\lambda,2\beta/3\}$;
the saving condition is $\mathcal Q(\nu)<\delta+\beta/2$.
If $\beta<13/20$, at the boundary $\delta=1-\beta/2$ it holds
because $\mathcal Q(2\beta/3)<1$. The function $\mathcal Q$
is nondecreasing and $1$-Lipschitz, so the saving cannot decrease
as $\delta$ increases. This proves the first branch of
\eqref{eq:zeta-threshold}.

For $\beta\ge13/20$, one has $\nu\ge13/30$, where the new
$\mathcal Q$ agrees with the former untruncated maximum.
If $\delta\le\delta_0:=1-\beta/3$, then $\nu=2\beta/3$;
substitution in \eqref{eq:q-function} gives the second and third
branches of \eqref{eq:zeta-threshold}. The resulting threshold
is at most $\delta_0$ precisely when $\beta\le\beta_*$, as follows
from $1+11\beta/6=2\sqrt{2\beta}$.
For $\beta>\beta_*$ the threshold has $\nu=\lambda$ in the
square-root branch. Solving
\[
 2+3\lambda-2\sqrt{3\lambda}-\beta/2=1-\beta+\lambda
\]
gives $\lambda=1-\beta/4-(\sqrt3/2)\sqrt{1-\beta}$ and the
last branch. The slope is strictly less than $1$ until the final
linear branch, whose saving is $(1-\beta)/2>0$.
At $\beta=13/20$, $\mathcal Q(2\beta/3)=1$, so the shorter
range supplies no strict saving; this explains the jump of
$d(\beta)$ from its left limit $17/30$ to $d(13/20)=27/40$.
The displayed branches give $\min d=17/30$.

For the frequency assertion, fix an admissible $\beta,\delta$,
write $w_0=1-\delta<13/30$, and regard $b=\beta+u$ as variable.
The exponent in the theorem is
\[
 E(u)=\mathcal Q\!\left(\max\{b-w_0,w_0,2b/3\}\right)
       -\max\{b/2,w_0\}.
\]
For $b\le3w_0/2$ it is constant. For $3w_0/2\le b\le2w_0$,
the argument of $\mathcal Q$ is $2b/3\le4w_0/3<26/45<16/27$.
On this range $\mathcal Q'\le3/4$, so $E'\le1/2$.
For $b\ge2w_0$, the subtracted term has derivative $1/2$,
and the first term has derivative at most $1$. Thus
$E(u)\le E(0)+u/2$, proving \eqref{eq:zeta-saving}.
\end{proof}

For $\beta=1/2$, $\delta=3/5$ and $0\le u\le1/10$,
$\nu_m=2/5$ and $a_0=2/5$. Hence
\[
 \mathcal Q(2/5)-2/5
 =\frac7{10}+\frac{18}{65}-\frac25=\frac{15}{26},
\]
which proves the example after Corollary~\ref{cor:zeta}.
Table~\ref{tab:zeta} follows by the same substitution.
At $\beta=1$, $\mathcal Q(\nu_m)\ge1/2+\nu_m$ implies
$\mathcal Q(\nu_m)-a_0\ge\delta$, so no strict saving follows.
For nonzero integral $\alpha$, the phase $e(\alpha n)$ is identically $1$.

The localization gain in this example uses more than the global
density bound. At $m=1$, the zero height is $K^{1/2}$ and the
window length is $K^{2/5}$. At $s=7/10$, the global density
bound gives the weighted exponent
$7/10+(1/2)(9/13)=68/65>1$, whereas our short-strip theorem gives
$7/10+(2/5)(9/13)=127/130<1$. Subtracting the transform exponent
$2/5$ gives $15/26$. Thus the global bound, even combined with
the $7/8$ cutoff and the local total-zero count, gives no saving
in this absolute-value argument, while the short-strip bound does.

We also compare with the earlier short-strip input, using the
same zero-free boundary in both cases. Ye--Zhang's bounds alone
replace $\mathcal Q(x)$ by
\[
 \mathcal Q_{\mathrm{old}}(x)=
 \max\left\{\frac78+\frac x4,\
       \max_{1/2\le s\le3/4}\left(s+\frac{3x(1-s)}{2-s}\right)\right\}.
\]
Its first two transitions are $5/14$ and $25/48$, with
intervening branch $3/4+3x/5$; the subsequent square-root and
linear branches agree with those of $\mathcal Q$. Consequently,
\begin{equation}\label{eq:zeta-density-gain}
 \mathcal Q(x)<\mathcal Q_{\mathrm{old}}(x)
 \quad\Longleftrightarrow\quad 5/14<x<169/300.
\end{equation}
For $\beta=2/3$, $\delta=17/25$, and
$1\le m\le K^{107/600}$, one has $\nu_m=4/9+2u/3$ in the
linear branch, so
\begin{equation}\label{eq:frequency-example}
 S_\zeta(m\alpha,2/3;K,K^{17/25})
       \ll K^{263/390+\varepsilon}m^{-1/26}.
\end{equation}
At $m=1$, the earlier short-strip bounds give $K^{41/60+\varepsilon}$
in the same argument, which does not save over $K^{17/25}$.

Finally, the zero-free boundary can be kept as a parameter.
For $6/7\le\vartheta\le11/12$ with no zeros to its right, replace
$\mathcal Q$ by $\mathcal Q_\vartheta$. Its first branch and
first transition are
\begin{equation}\label{eq:boundary-refinement}
 \mathcal Q_\vartheta(x)=\vartheta+2(1-\vartheta)x
 \quad(0\le x\le x_\vartheta),\qquad
 x_\vartheta=\frac{\vartheta-7/10}{2\vartheta-17/13};
\end{equation}
the other branches are those in \eqref{eq:q-function}, starting
the next branch at $x_\vartheta$. This follows by the same endpoint
comparison. The boundary $11/12$ of \cite{OAI11} gives
$x_\vartheta=169/410$. Liu's boundary \cite[Theorem~27.1]{BLiu} is
\[
 \begin{gathered}
 \vartheta_L=\frac{18\xi+7}{18\xi+9}=0.874957019420098946\ldots,\\
 1188\xi^3-2160\xi^2+171\xi+190=0,
 \end{gathered}
\]
where $0.38858<\xi<0.38859$ specifies the root uniquely.
It slightly improves the first branch and its transition.
For every such $\vartheta<1$, however,
$\mathcal Q_\vartheta(x)<1$ exactly when $x<13/30$;
thus the admissible ranges in Corollary~\ref{cor:zeta} are unchanged.

\section{\texorpdfstring{$L$-functions for cusp forms for $\Gamma_0(N)$}
 {L-functions for cusp forms for Gamma0(N)}}\label{sec:gl2}

\begin{proof}[Proof of Theorem~\ref{thm:gl2-sum}]
Write $\pi=g\otimes\chi$ and $I=I_m$. Since $L(s,\pi)$ is entire, its
explicit formula has no pole terms. Proposition~\ref{prop:frequency-reduction}
reduces the estimate to the weighted zero sum in
\eqref{eq:frequency-reduction}. The covering argument of
Section~\ref{sec:zeta} now uses $a=v/b\in(1/2,1)$,
so no enlargement of the length exponent is needed. Self-duality
allows reflection if $I$ lies at negative height.

Put $C=b+3v$ and $D=b+v$.
Ye--Zhang \cite[(2.6)--(2.7)]{YZ}, together with the local bound
$\#\{\rho_\pi:\Im\rho_\pi\in I\}\ll K^{v+\varepsilon}$,
give
\begin{equation}\label{eq:gl2-density-window}
 N_\pi(s;I)\ll K^{\mathcal D_2(s)+\varepsilon},\qquad
 \mathcal D_2(s)=
 \begin{cases}
 \displaystyle\min\!\left\{v,\frac{C(1-s)}{2-s}\right\},
                                     &1/2\le s<3/4,\\[5pt]
 \displaystyle\frac{D(1-s)}s,        &3/4\le s<1.
 \end{cases}
\end{equation}
On the upper band, $D(1-s)/s\le D/3<v$.
Partial summation as in \eqref{eq:zeta-partial-summation}, with
$\zeta$ replaced by $\pi$ and upper endpoint $1$, gives
\begin{equation}\label{eq:gl2-weighted-max}
 S_\pi(m\alpha,\beta;K,K^\delta)\ll K^{-b/2+\sup_{1/2\le s<1}\{s+\mathcal D_2(s)\}
                      +\varepsilon}.
\end{equation}
The partial summation identity includes critical-line zeros and zeros at $s=3/4$;
the lower-band supremum may be a limit from the left.

The entries in the minimum in \eqref{eq:gl2-density-window}
meet at
\[
 c=\frac{b+v}{b+2v}\in(2/3,3/4).
\]
Thus $s+\mathcal D_2(s)=s+v$ on $[1/2,c]$, where it is
increasing. On $[c,3/4)$ it equals the concave function
\[
 G(s)=s+\frac{C(1-s)}{2-s},\qquad G'(s)=1-\frac C{(2-s)^2}.
\]
The lower-band supremum is therefore $G(s_*)$, where
\begin{equation}\label{eq:gl2-optimizer}
 s_* =\min\{3/4,\max\{c,2-\sqrt C\}\}.
\end{equation}
The transitions occur at $C=25/16$ and
$(b+2v)^2=C$, giving exactly
$G(s_*)=\mathcal F_2(b,v)$ in \eqref{eq:f2}.
On the upper band, $s+D(1-s)/s$ is convex and has supremum
$\max\{1,3/4+D/3\}$. Finally,
\[
 \mathcal F_2(b,v)\ge\frac34+\frac C5
       >\frac34+\frac D3,
 \qquad \frac C5-\frac D3=\frac{4v-2b}{15}>0.
\]
The supremum in \eqref{eq:gl2-weighted-max} is consequently
$\max\{1,\mathcal F_2(b,v)\}$, proving
\eqref{eq:gl2-sum}.
\end{proof}

\begin{proof}[Proof of Corollary~\ref{cor:gl2}]
First set $m=1$. Since $\lambda>\beta/2$, the entry $1$ in \eqref{eq:gl2-saving}
causes no obstruction. The saving condition is
\begin{equation}\label{eq:gl2-saving-condition}
 \mathcal F_2(\beta,\lambda)<1-\beta/2+\lambda.
\end{equation}
On each applicable branch, $\partial\mathcal F_2/\partial\lambda<1$.
Indeed, the three derivatives are $3/5$, $3-3/\sqrt C$, and
$1-\beta/(\beta+2\lambda)^2$; in the middle branch
$\sqrt C\le2-c<4/3$. Thus the saving increases with $\lambda$.
For $0<\beta\le1/2$, at the limiting endpoint $\lambda=\beta/2$
one has $\mathcal F_2=3/4+\beta/2$. This proves the first branch of
\eqref{eq:gl2-threshold}.

For $\beta>1/2$, solve equality in \eqref{eq:gl2-saving-condition}
on the three branches of \eqref{eq:f2}, and then use
$\delta=1-\beta+\lambda$. The first gives
$\delta=3/8+3\beta/4$ and reaches $C=25/16$ at $\beta=11/20$.
In the middle branch the equation becomes
$4(\sqrt C-3/2)^2=3-5\beta$, giving the third entry of
\eqref{eq:gl2-threshold}. It reaches
$(\beta+2\lambda)^2=C$ when $\beta^2+3\beta-2=0$,
that is, at $\beta_c$. The last branch reduces to
$\lambda>\beta^2/(2(1-\beta))$ and gives the fourth entry.
The branches agree at their endpoints; the first decreases to
$3/4$ and the remaining branches increase from $3/4$ to $1$.

There is no further saving range for $\beta\ge2/3$:
$\mathcal F_2\ge c+\lambda$ implies
\[
 \omega_2\le\frac{\lambda}{\beta+2\lambda}-\frac\beta2
             <\frac13-\frac\beta2\le0.
\]
This proves both the exact range and the minimum.
The uniform assertion follows from the frequency comparison
at the end of Section~\ref{sec:gl3}.
\end{proof}

For $\beta=1/2$ and $3/4<\delta\le4/5$, the first branch of
\eqref{eq:f2} gives exponent $3/10+3\delta/5$.
This, $d_2(1/2)=3/4$, and the frequency comparison prove
the example after Corollary~\ref{cor:gl2};
the entries of Table~\ref{tab:gl2} follow by substitution.

\section{\texorpdfstring{$L$-functions for Maass forms for $\SL_3(\Z)$}
 {L-functions for Maass forms for SL3(Z)}}\label{sec:gl3}

\begin{proof}[Proof of Theorem~\ref{thm:gl3-sum}]
Write $I=I_m$. Again $L(s,f)$ is entire, so no pole terms occur, and
Proposition~\ref{prop:frequency-reduction} applies.
Put
\[
 C=2b+4v,\qquad D=2(b+v),\qquad
 c=\frac{2b+v}{2b+2v}
     =1-\frac{v}{D}\in(3/4,5/6).
\]
With $a=v/b\in(1/2,1)$, Ye--Zhang
\cite[(2.10)--(2.11)]{YZ} and the same bounded covering give
\begin{equation}\label{eq:gl3-density-window}
 N_f(s;I)\ll
 \begin{cases}
 K^{C(1-s)/(3-2s)+\varepsilon},&1/2\le s<c,\\[3pt]
 K^{D(1-s)+\varepsilon},&c\le s<1.
 \end{cases}
\end{equation}
The local bound $N_f(s;I)\ll K^{v+\varepsilon}$ is
stronger throughout the lower band, since
\[
 \frac{C(1-s)}{3-2s}-v
       =\frac{D(c-s)}{3-2s}\ge0\qquad(1/2\le s\le c).
\]
On the upper band, $D(1-s)\le D(1-c)=v$.
Thus the effective density bound is
\begin{equation}\label{eq:gl3-effective-density}
 N_f(s;I)\ll K^{\mathcal D_3(s)+\varepsilon},\qquad
 \mathcal D_3(s)=\min\{v,D(1-s)\}.
\end{equation}
Partial summation as in \eqref{eq:zeta-partial-summation}, with upper
endpoint $1$, now reduces the weighted sum to two linear functions: $s+v$
on $[1/2,c]$ and $s+D(1-s)$ on $[c,1]$. Hence
\begin{equation}\label{eq:gl3-weighted-max}
 \sup_{1/2\le s<1}\{s+\mathcal D_3(s)\}
       =\max\{c+v,1\}.
\end{equation}
Inserting this in \eqref{eq:frequency-reduction} proves
\eqref{eq:gl3-sum}, since
$c+v-b/2=1+v-b/2-v/(2(b+v))$.
The first term governs when $D\le1$, and the second when
$D\ge1$, because $c+v-1=v(1-1/D)$.
\end{proof}

\begin{proof}[Proof of Corollary~\ref{cor:gl3}]
First set $m=1$. The two entries in \eqref{eq:gl3-saving} are the savings in the
two terms of \eqref{eq:gl3-sum}. For $0<\beta<1$, both are
positive precisely when
\[
 \lambda>\max\left\{\frac\beta2,\frac{\beta^2}{1-\beta}\right\}.
\]
The entries cross at $\beta=1/3$; the second is less than
$\beta$ precisely when $\beta<1/2$. Substituting
$\delta=1-\beta+\lambda$ gives \eqref{eq:gl3-threshold}.
Its first branch decreases to $5/6$, and its second increases
from $5/6$ to $1$. For $1/2\le\beta\le1$,
$\lambda/(2(\beta+\lambda))<1/4\le\beta/2$, so no strict saving
is supplied by \eqref{eq:gl3-sum}. The uniform assertion follows
from the frequency comparison below.
\end{proof}

For $\beta=1/3$, the second entry of \eqref{eq:gl3-saving}
is the smaller and equals $(\delta-5/6)/(3\delta-1)$.
Together with the frequency comparison, this proves
the example after Corollary~\ref{cor:gl3};
substitution gives Table~\ref{tab:gl3}.
Both cusp-form arguments use Ye--Zhang's density bounds and
the local zero count, without a Ramanujan hypothesis.

To complete the uniform assertions in Corollaries~\ref{cor:gl2}
and \ref{cor:gl3}, write
$\mathcal D_j(s;b,v)$ for the density envelopes in
\eqref{eq:gl2-density-window} and \eqref{eq:gl3-effective-density}.
For $j=2,3$,
\[
 \mathcal D_j(s;\beta+u,\lambda+u)
       \le\mathcal D_j(s;\beta,\lambda)+u.
\]
For $j=2$, below $s=2/3$ the local-count entry governs; above
$2/3$ the lower-band density slope in $u$ is
$4(1-s)/(2-s)\le1$, and the upper-band slope is $2(1-s)/s\le2/3$.
For $j=3$, use the local-count entry below its old crossing
$c>3/4$ and the slope $4(1-s)<1$ above it. Taking the weighted
supremum and including $K^{-u/2}$ gives the factor $m^{1/2}$
in both corollaries.

\vspace{0.125in}
\noindent
{\bf Funding  Sources}

\noindent
This research did not receive any specific 
grant from funding agencies in the public, 
commercial, or not-for-profit sectors.

\vspace{0.125in}
\noindent
{\bf Declaration of generative AI and 
AI-assisted technologies in the manuscript 
preparation process.}

\noindent
During the preparation of this work, the 
author used OpenAI's GPT-6 Astra to 
implement proof strategy, assist 
with computations, support literature 
searches, and typeset the manuscript. 
The author carefully reviewed, 
verified, and edited all AI-generated 
content as necessary and takes full 
responsibility for the accuracy, 
originality, and integrity of the 
published article.

\end{document}